\documentclass[12pt]{amsart}
\usepackage{amsmath}
\usepackage{amsfonts}

\newtheorem{theorem}{Theorem}

\newtheorem{corollary}[theorem]{Corollary}

\newtheorem{definition}{Definition}
\newtheorem{example}{Example}

\newtheorem{lemma}[theorem]{Lemma}

\newtheorem{proposition}[theorem]{Proposition}
\newtheorem{remark}{Remark}

\begin{document}

\title[Multi-anisotropic Gevrey regularity]
 {Multi-anisotropic Gevrey regularity\\of hypoelliptic operators}
\author[C. Bouzar ]{Chikh Bouzar}

\address{%
Oran-Essenia University, Algeria}

\email{bouzar@yahoo.com}

\thanks{}

\author[A. Dali]{Ahmed Dali}
\address{
University of Bechar, Algeria}
 \email{Ahmedalimat@yahoo.fr}



\date{}

\begin{abstract}
We show a multi-anisotropic Gevrey regularity of solutions of
hypoelliptic equations. This result is a precision of a classical
result of H\"{o}rmander
\end{abstract}

\maketitle

\section{Introduction}

An important problem among others of linear partial differential
equations is the $\mathcal{C}^{\infty }$ or Gevrey regularity of
solutions of these
equations. L. H\"{o}rmander has completely characterized the $\mathcal{C}%
^{\infty }$ regularity (hypoellipticity) of linear partial
differential operators with complex constant coefficients, see
\cite{HOR2}. An another fundamental result obtained by L.
H\"{o}rmander says that every hypoelliptic differential operator
$P\left( D\right) $ is anisotropic Gevrey hypoelliptic, i.e.
$\exists \varrho =\left( \varrho _{1},....,\varrho _{n}\right) \in
\mathbb{R}_{+}^{n}$ such that
\begin{equation*}
\text{ }u\in \mathcal{D}^{\prime }\left( \Omega \right) \text{ }and\text{ }%
P(D)u=0\Longrightarrow \text{ }u\in G^{\varrho }\left( \Omega
\right) ,
\end{equation*}%
where $G^{\varrho }\left( \Omega \right) $ is an anisotropic
Gevrey space associated with $P\left( D\right) .$

A large class of hypoelliptic differential operators is the class
of
multi-quasielliptic differential operators, see V. P. Mikha\"{\i}lov \cite%
{MIKH}, J. Friberg \cite{FRIB} and S. G. Gindikin, L. R. Volevich
\cite{GV}.

L. Zanghirati \cite{ZANG}, proved that multi-quasielliptic
differential
operators are multi-anisotropic Gevrey hypoelliptic, i.e.%
\begin{equation*}
\text{ }u\in \mathcal{D}^{\prime }\left( \Omega \right) \text{ }and\text{ }%
P(D)u\in G^{s,\text{ }\Gamma }\left( \Omega \right) \text{
}\Longrightarrow \text{ }u\in G^{s,\text{ }\Gamma }\left( \Omega
\right) ,
\end{equation*}%
where $G^{s,\text{ }\Gamma }\left( \Omega \right) $ is a Gevrey
multi-anisotropic space associated with $P\left( D\right) $. This
result clarifies the classical result of H\"{o}rmander in the case
of multi-quasielliptic operators. The result of L. Zanghirati has
been extended by C. Bouzar and R. Cha\"{\i}li in \cite{BC2} to
multi-quasielliptic systems of differential operators.

The aim of this paper is to prove the multi-anisotropic Gevrey
regularity of hypoelliptic linear differential operators with
complex constant coefficients, and consequently we precise the
result of H\"{o}rmander and extend the result of Zanghirati.

\section{Multi-quasiellipticity}

Let $\Omega $ be an open subset of $\mathbb{R}^{n}$, if $\alpha
=\left( \alpha _{1},....\alpha _{n}\right) \in
\mathbb{Z}_{+}^{n},$ $q=\left( q_{1},..,q_{n}\right) \in
\mathbb{R}_{+}^{n}$ and $\xi =\left( \xi _{1},...,\xi _{n}\right)
$ $\in \mathbb{R}^{n}$, we set
\begin{equation*}
\left\vert \alpha \right\vert =\alpha _{1}+...+\alpha _{n}
\end{equation*}

\begin{equation*}
<\alpha ,q>=\sum_{j=1}^{n}\alpha _{j}q_{j}
\end{equation*}%
\begin{equation*}
\xi ^{\alpha }=\xi _{1}^{\alpha _{1}}...\text{ }\xi _{n}^{\alpha
_{n}}
\end{equation*}%
\begin{equation*}
D^{\alpha }=D_{1}^{\alpha _{1}}...\text{ }D_{n}^{\alpha _{n}},D_{j}=\dfrac{1%
}{i}\dfrac{\partial }{\partial \xi _{j}},\text{ }j=1,...,n.
\end{equation*}%
\begin{equation*}
\mathbb{R}_{+}^{n}=\left\{ \xi \in \mathbb{R}^{n}:\xi _{j}>0,\text{ }%
j=1,..,n\right\}
\end{equation*}%
The space $\mathcal{C}_{0}^{\infty }\left( \Omega \right) $ is the
space of functions $u\in \mathcal{C}^{\infty }$ with compact
support in $\Omega $. The space of distributions on $\Omega $ is
denoted $\mathcal{D}^{\prime }\left( \Omega \right) .$

\begin{definition}
Let $A$ be a finite subset of $\overline{\mathbb{R}_{+}^{n}}$, the
Newton's
polyhedron of $A$, denoted $\Gamma (A\mathbb{)}$, is the convex hull of $%
\left\{ 0\right\} \cup A.$
\end{definition}

A Newton's polyhedron $\Gamma $ is always characterized by
\begin{equation*}
\Gamma \mathbb{=}\underset{q\in \mathcal{A}\left( \Gamma \right) }{\bigcap }%
\left\{ \alpha \in \overline{\mathbb{R}_{+}^{n}},\text{
}\left\langle q,\alpha \right\rangle \leq 1\right\} ,
\end{equation*}%
where $\mathcal{A}\left( \Gamma \right) $ is a finite subset of $\mathbb{R}%
^{n}\ \left\{ 0\right\} $.

\begin{definition}
Let $\Gamma \mathbb{=}\underset{q\in \mathcal{A}\left( \Gamma \right) }{%
\bigcap }\left\{ \alpha \in \overline{\mathbb{R}_{+}^{n}},\text{ }%
\left\langle q,\alpha \right\rangle \leq 1\right\} $ be a Newton's
polyhedron, $\Gamma $ is said to be regular, if
\begin{equation*}
q_{j}>0,\text{ }\forall j=1,...,n;\text{ }\forall q=\left(
q_{1},...,q_{n}\right) \in \mathcal{A}\left( \Gamma \right)
\end{equation*}
\end{definition}

We associate with a regular Newton's polyhedron $\Gamma $ the
following elements
\begin{equation*}
\mathcal{V}\left( \Gamma \right) =\left\{ s^{0}=0,\text{ }%
s^{1},...,s^{m\left( \Gamma \right) }\right\} \text{ the set of vertices of }%
\Gamma
\end{equation*}%
\begin{equation*}
\left\vert \xi \right\vert _{\Gamma }=\sum\limits_{\nu \in
\mathcal{V}\left(
\Gamma \right) }\left\vert \xi \right\vert ^{\nu },\text{ }\xi \in \mathbb{R}%
^{n},\text{ where }\left\vert \xi \right\vert ^{\nu }=\left\vert
\xi _{1}\right\vert ^{\nu _{1}}...\left\vert \xi _{n}\right\vert
^{\nu _{n}}
\end{equation*}%
\begin{equation*}
k\left( \alpha ,\Gamma \right) =\inf \left\{ t>0\text{,
}t^{-1}\alpha \in
\Gamma \right\} =\underset{q\in \mathcal{A}\left( \Gamma \right) }{\max }%
\left\langle \alpha ,q\right\rangle ,\text{ }\alpha \in \mathbb{R}_{+}^{n}%
\text{ }
\end{equation*}%
\begin{equation*}
\mu \left( \Gamma \right) =\underset{%
\begin{array}{c}
q\in \mathcal{A}\left( \Gamma \right)  \\
1\leq \text{ }j\leq n%
\end{array}%
}{\max }q_{j}^{-1}\text{ called the formal order of }\Gamma
\end{equation*}

A differential operators with complex constant coefficients
\begin{equation*}
P\left( D\right) =\underset{\alpha }{\sum }a_{\alpha }D^{\alpha }
\end{equation*}%
has its complete symbol
\begin{equation*}
P\left( \xi \right) =\underset{\alpha }{\sum }a_{\alpha }\xi
^{\alpha }
\end{equation*}

\begin{definition}
The Newton's polyhedron of $P$, denoted $\Gamma (P\mathbb{)}$, is
the convex
hull of the set $\left\{ 0\right\} \cup \left\{ \alpha \in \mathbb{Z}%
_{+}^{n}:\text{ }a_{\alpha }\neq 0\right\} .$
\end{definition}

Define the weight function
\begin{equation*}
\left\vert \xi \right\vert _{\mathbb{P}}=\sum\limits_{\alpha \in \mathcal{V}%
\left( P\right) }\left\vert \xi ^{\alpha }\right\vert ,\text{
}\forall \xi \in \mathbb{R}^{n},
\end{equation*}%
where $\mathcal{V}\left( P\right) =\mathcal{V}\left( \Gamma \left(
P\right)
\right) $ is the set of vertices of $\Gamma \left( P\right) .$ Recall%
\begin{equation*}
d(\xi ):=dist(\xi ,N(P)),\text{ where }N(P):=\left\{ \zeta \in \mathbb{C}%
^{n}:P(\zeta )=0\right\}
\end{equation*}

\begin{definition}
The differential operator $P(D)$ is said hypoelliptic in $\Omega
$, if
\begin{equation*}
singsuppP(D)u=singsuppu,\text{ }\forall u\in \mathcal{D}^{\prime
}(\Omega )
\end{equation*}
\end{definition}

The characterization of hypoelliptic differential operators with
constant coefficients is du to L. H\"{o}rmander. The following
result, see the theorem 4.1.3 of \cite{HOR2}, gives some
characterizations of the hypoellipticity.

\begin{theorem}
Let $P(D)$ be a differential operator with constant coefficients,
the following properties are equivalent :

i) The operator $P\left( D\right) $ is hypoelliptic.

ii) $\exists C>0,$ $\exists d>0,$ $\left\vert \xi \right\vert
^{d}\leq Cd(\xi ),$ $\forall \xi \in \mathbb{R}^{n},$ $\left\vert
\xi \right\vert $ large.

iii) If $\xi \in \mathbb{R}^{n},\left\vert \xi \right\vert
\rightarrow
+\infty ,$ then $\dfrac{\left\vert D^{\alpha }P(\xi )\right\vert }{%
\left\vert P(\xi )\right\vert }\rightarrow 0,$ $\forall \alpha
\neq 0.$

iv) $\exists C>0,$ $\exists \rho >0,$ $\dfrac{\left\vert D^{\alpha
}P(\xi )\right\vert }{\left\vert P(\xi )\right\vert }\leq
C\left\vert \xi \right\vert ^{-\rho \left\vert \alpha \right\vert
},$ $\forall \xi \in \mathbb{R}^{n},$ $\left\vert \xi \right\vert
$ large.
\end{theorem}

The connection between an hypoelliptic operator and its Newton's
polyhedron is given by the following proposition.

\begin{proposition}
The Newton's polyhedron of an hypoelliptic differential operator
is regular.
\end{proposition}

\begin{proof}
See \cite{FRIB}$.$
\end{proof}

\begin{remark}
The converse is not true, $\square =D_{x}^{2}-D_{y}^{2}$ has a
regular Newton's polyhedron with vertices $\left\{ \left(
0,0\right) ,(2,0),(0,2)\right\} ,$ but the operator $\square $ is
not hypoelliptic.
\end{remark}

We introduce multi-quasielliptic polynomials which are a natural
generalization of the classical quasi-elliptic operators. These
operators were characterized first by V. P. Mikha\"{\i}lov
\cite{MIKH}, then studied by J. Friberg \cite{FRIB} and finally
far developed by S. G. Gindikin and L. R. Volevich \cite{GV}$.$

\begin{definition}
\bigskip The polynomial $P(\xi )=\sum\limits_{\alpha }a_{\alpha }\xi
^{\alpha }$ is said to be multi-quasielliptic, if

\begin{enumerate}
\item[i)] its Newton's polyhedron $\Gamma \left( P\right) $ is
regular.

\item[ii)] $\exists C>0$ such that
\begin{equation*}
\left\vert \xi \right\vert _{\mathbb{P}}\leq C(1+\left\vert P(\xi
)\right\vert ),\text{ }\forall \xi \in \mathbb{R}^{n}
\end{equation*}
\end{enumerate}
\end{definition}

\begin{proposition}
A multi-quasielliptic operator $P\left( D\right) $ is
hypoelliptic.
\end{proposition}

\begin{proof}
See \cite{FRIB} or \cite{GV}.
\end{proof}

\begin{remark}
The converse is not true. Indeed, consider the following
polynomial
\begin{eqnarray*}
P\left( \xi ,\eta \right) &=&i\xi ^{5}+i\xi \eta ^{4}-4i\xi
^{4}\eta -4i\xi ^{2}\eta ^{3}+6i\xi ^{3}\eta ^{2}+i\xi ^{3}+i\xi
\eta ^{2}+\xi ^{4}\eta ^{2}
\\
&&+\eta ^{6}-4\xi ^{3}\eta ^{3}-4\xi \eta ^{5}+6\xi ^{2}\eta
^{4}+\eta ^{2}\xi ^{2}+\eta ^{4},
\end{eqnarray*}%
which is hypoelliptic thanks to the theorem 4.1.9 of \cite{HOR2}.
We have
\begin{equation*}
P_{\left( 1,1\right) }\left( \xi ,\eta \right) =\eta ^{2}\left(
\xi ^{4}+\eta ^{4}-4\xi ^{3}\eta -4\xi \eta ^{3}+6\xi ^{2}\eta
^{2}\right) =\eta ^{2}\left( \xi -\eta \right) ^{4}
\end{equation*}%
The $q=\left( 1,1\right) -$quasiprincipal part of $P\left( \xi
,\eta \right) $ degenerates on the straight $\xi =\eta $, hence
the polynomial $P\left( \xi ,\eta \right) $ is not
multi-quasielliptic, see \cite{GV}.
\end{remark}

\section{Multi-anisotropic Gevrey vectors}

The multi-anisotropic Gevrey spaces were explicitly defined by L.
Zanghirati in \cite{ZANG} for studying the multi-anistropic Gevrey
regularity of multi-quasielliptic differential operators by the
method of elliptic iterates.\bigskip

\begin{definition}
Let $\Omega $ be an open subset of $\mathbb{R}^{n}$, $\Gamma $ a
regular Newton's polyhedron and $s\geq 1.$ Denote $G^{s,\text{
}\Gamma }(\Omega )$
the space of functions $u\in \mathcal{C}^{\infty }(\Omega )$ such that $%
\forall K\subset \Omega ,\exists C>0$, $\forall \alpha \in \mathbb{Z}%
_{+}^{n} $,
\begin{equation}
\underset{x\in K}{\sup }\left\vert D^{\alpha }u(x)\right\vert \leq
C^{\left\vert \alpha \right\vert +1}k(\alpha ,\Gamma )^{s\mu
\;k(\alpha ,\Gamma )}  \label{4.1.1}
\end{equation}
\end{definition}

\begin{example}
If $\Gamma $ is the regular Newton's polyhedron defined by
\begin{equation*}
\Gamma \mathbb{=}\left\{ \alpha \in \overline{\mathbb{R}_{+}^{n}}:\underset{%
j=1}{\overset{n}{\sum }}m_{j}^{-1}\alpha _{j}\leq 1,m_{j}\in \mathbb{R}%
_{+}\right\} ,
\end{equation*}%
then
\begin{equation*}
G^{s,\text{ }\Gamma }\left( \Omega \right) =\left\{
\begin{array}{c}
u\in \mathcal{C}^{\infty }\left( \Omega \right) ,\forall K\subset
\Omega
,\exists C>0,\forall \alpha \in \mathbb{Z}_{+}^{n} \\
\left\vert D^{\alpha }u\left( x\right) \right\vert \leq
C^{\left\vert \alpha
\right\vert +1}\langle \alpha ,q\rangle ^{s\langle \alpha ,q\rangle }%
\end{array}%
\right\} ,
\end{equation*}%
where $q:=\left( \dfrac{m}{m_{1}},...,\dfrac{m}{m_{n}}\right) $ and $m:=%
\underset{j}{\max }m_{j}$, i.e. $G^{s,\text{ }\Gamma }\left(
\Omega \right) $ is the classical anisotropic Gevrey space
$G^{s,\;q}\left( \Omega \right) $.
If $m_{1}=m_{2}=...=m_{n}$, we obtain the classical isotropic Gevrey space $%
G^{s}\left( \Omega \right) .$
\end{example}

\begin{definition}
Let $\Gamma $ be the regular Newton's polyhedron of $P\left( D\right) $ and $%
s\geq 1,$ the space of Gevrey vectors of $P\left( D\right) $, denoted $%
G^{s}\left( \Omega ,P\right) $, is the space of $u\in
\mathcal{C}^{\infty
}\left( \Omega \right) $ such that , $\forall K$ compact of $\Omega $, $%
\exists C>0,\forall l\in \mathbb{N},$%
\begin{equation*}
\left\Vert P^{l}u\right\Vert _{L^{\infty }\left( K\right) }\leq
C^{l+1}\left( l!\right) ^{s\mu \left( \Gamma \right) }
\end{equation*}
\end{definition}

\begin{remark}
We can take $l^{sl\mu \left( \Gamma \right) }$\ instead of $\left(
l!\right) ^{s\mu \left( \Gamma \right) }.$
\end{remark}

We recall a result of L. Zanghirati \cite{ZANG}\ and C. Bouzar and R. Cha%
\"{\i}li \cite{BC2} wich gives the multi-anisotropic Gevrey
regularity of Gevrey vectors of multi-quasielliptic operators.

\begin{theorem}
Let $\Omega $ be an open subset of $\mathbb{R}^{n},$ $s>1$ and $P$
a linear differential operator with complex constant coefficients
with regular Newton's polyhedron $\Gamma $. Then the following
assertions are equivalent :

$i)$ $P$ is multi-quasielliptic in $\Omega $

$ii)$ $G^{s}\left( \Omega ,P\right) =G^{s,\text{ }\Gamma }\left(
\Omega \right) $
\end{theorem}

\section{Multi-anisotropic Gevrey hypoellipticity of hypoelliptic operators}

In this section, $P=\underset{\alpha }{\sum }a_{\alpha }D^{\alpha
}$ is an hypoelliptic differential operator with complex constant
coefficients.

\begin{definition}
A finite set $\mathcal{H}$ $\subset \overline{\mathbb{R}_{+}^{n}}$
is said a polyhedron of hypoellipticity of $P,$ if

\begin{enumerate}
\item $\forall \nu \in \mathcal{H},\exists C>0,\forall \xi \in \mathbb{R}%
^{n},\;\left\vert \xi \right\vert ^{\nu }\leq C\left( 1+d\left(
\xi \right) \right) $

\item $\mathcal{H}$ has vertices with rational components.

\item $\mathcal{H}$ is regular.
\end{enumerate}
\end{definition}

\begin{remark}
If $\nu $ belongs to the convex hull of $\mathcal{H},$ i.e. $\nu
=\sum\limits _{\substack{ i\in I  \\ I\;fini}}\lambda _{i}\beta
_{i},$ where $\beta _{i}\in \mathcal{H}$ and $\sum\limits_{i\in
I}\lambda _{i}=1,\lambda _{i}\geq 0,$ then $\left\vert \xi
\right\vert ^{\nu }\leq C\left( 1+d(\xi )\right) ,$ $\forall \xi
\in \mathbb{R}^{n},$ therefor it is natural to assume that
$\mathcal{H}$ is convex.
\end{remark}

\begin{remark}
The set $\mathcal{H}$ is never empty, as an hypoelliptic operator
satisfies : $\exists C>0,$ $\exists d>0,\left\vert \xi \right\vert
^{d}\leq C\left( 1+d\left\vert \xi \right\vert \right) ,$ $\forall
\xi \in \mathbb{R}^{n}.$
\end{remark}

\begin{definition}
Denote $\sigma $ be the smallest natural integer such that
\begin{equation*}
\sigma \mathcal{V}\left( \mathcal{H}\right) \subset
2\mathbb{N}_{0}^{n},
\end{equation*}%
and define the differential operator $Q_{\mathcal{H}}$ $\left(
D\right) ,$ by
\begin{equation*}
Q_{\mathcal{H}}\left( D\right) =\sum_{\alpha \in \mathcal{V}\left( \mathcal{H%
}\right) }D^{\sigma \alpha }
\end{equation*}
\end{definition}

\begin{proposition}
The operator $Q_{\mathcal{H}}\left( D\right) $ is
multi-quasielliptic.
\end{proposition}

\begin{proof}
The Newton's polyhedron of the differential operator
$Q_{\mathcal{H}}$ has
vertices with even positive integer components. Then%
\begin{equation*}
\left\vert Q_{\mathcal{H}}\left( \xi \right) \right\vert
=\underset{\alpha \in \mathcal{V}\left( \mathcal{H}\right) }{\sum
}\left\vert \xi ^{\sigma \alpha }\right\vert =\left\vert \xi
\right\vert _{Q_{\mathcal{H}}},
\end{equation*}%
hence%
\begin{equation*}
1+\left\vert \xi \right\vert _{Q_{\mathcal{H}}}\leq \left( 1+\left\vert Q_{%
\mathcal{H}}\left( \xi \right) \right\vert \right) ,\text{
}\forall \xi \in \mathbb{R}^{n}
\end{equation*}
\end{proof}

Let $v\in \mathcal{C}_{0}^{\infty }\left( \mathbb{R}^{n}\right)
,s\in \mathbb{Z}_{+}$ and $\varepsilon >0,$ then
\begin{equation*}
\left\vert \left\vert \left\vert v\right\vert \right\vert
\right\vert _{s,\varepsilon }^{2}:=\int_{\mathbb{R}^{n}}\left(
1+\varepsilon d\left( \xi \right) \right) ^{s}\left\vert
\widehat{v}\left( \xi \right) \right\vert ^{2}d\xi
\end{equation*}%
The following result is the lemma 4.4.3 of \cite{HOR2}$.$

\begin{lemma}
Let $u$\ be a solution of the equation $Pu=0$ defined in the ball $%
B_{\varepsilon }=\left\{ x\in \mathbb{R}^{n}:\left\vert
x\right\vert <\varepsilon \right\} ,$ and let $\varphi \in
\mathcal{C}_{0}^{\infty }\left( B_{1}\right) $ and the integer
$s\geq 1.$ Then
\begin{equation}
\underset{\alpha \neq 0}{\sum }\varepsilon ^{-2\left\vert \alpha
\right\vert }\left\vert \left\vert \left\vert P^{\left( \alpha
\right) }\left( D\right) \left( \varphi ^{\varepsilon }u\right)
\right\vert \right\vert \right\vert _{s,\varepsilon }^{2}\leq
C\underset{\alpha \neq 0}{\sum }\varepsilon ^{-2\left\vert \alpha
\right\vert }\int_{B_{\varepsilon }}\left\vert P^{\left( \alpha
\right) }\left( D\right) u\right\vert ^{2}dx,
\end{equation}%
where $C$ is independent of $\varepsilon $ and $u.$
\end{lemma}

\begin{remark}
In the lemma $\varphi ^{\varepsilon }$\ denotes $\varphi
^{\varepsilon }\left( x\right) :=\varphi \left(
\frac{x}{\varepsilon }\right) $ .
\end{remark}

Thanks to this lemma, we obtain the following result$.$

\begin{lemma}
Let $\beta $ $\in \mathbb{Z}_{+}^{n}\cap \sigma \mathcal{H}$, then
there
exists a constant $C>0,$ such that for every solution $u$ of $Pu=0$ in $%
B_{\varepsilon }$ and $\varepsilon \in \left] 0,1\right[ ,$ we
have
\begin{equation*}
\varepsilon ^{2\sigma }\underset{\alpha \neq 0}{\sum }\varepsilon
^{-2\left\vert \alpha \right\vert }\underset{B_{\frac{\varepsilon }{2}}}{%
\int }\left\vert P^{\left( \alpha \right) }\left( D\right)
D^{\beta }u\right\vert ^{2}dx\leq C\underset{\alpha \neq 0}{\sum
}\varepsilon
^{-2\left\vert \alpha \right\vert }\underset{B_{\varepsilon }}{\int }%
\left\vert P^{\left( \alpha \right) }\left( D\right) u\right\vert
^{2}dx
\end{equation*}
\end{lemma}

\begin{proof}
Let $\beta $ $\in \mathbb{Z}_{+}^{n}\cap \sigma \mathcal{H}$, from
$(1)$ of
definition $4.1$, we have%
\begin{equation*}
\left\vert \xi ^{\beta }\right\vert \leq C^{\sigma }\left(
1+d\left( \xi \right) \right) ^{\sigma },
\end{equation*}%
hence $\exists C>0$, $\forall \varepsilon \in \left] 0,1\right[ ,$
$\forall \xi \in \mathbb{R}^{n},$
\begin{equation}
\varepsilon ^{\sigma }\left\vert \xi ^{\beta }\right\vert \leq
C^{\sigma }d_{\sigma ,\varepsilon }\left( \xi \right)
\end{equation}%
Multiplying $(4.2)$ by $\left( 2\pi \right) ^{-n}\left\vert \widehat{v}%
\left( \xi \right) \right\vert $ and integrating with respect to
$\xi ,$ we
obtain%
\begin{equation}
\varepsilon ^{2\sigma }\int \left\vert D^{\beta }v\right\vert
^{2}dx\leq C^{2}\left\vert \left\vert \left\vert v\right\vert
\right\vert \right\vert _{\sigma ,\varepsilon }^{2}
\end{equation}%
Let $\varphi \in C_{0}^{\infty }\left( B_{1}\right) $ equals $1$ in $B_{%
\frac{1}{2}}$ and apply the estimate $(4.3)$ to $v=P^{\left(
\alpha \right)
}\left( D\right) \left( \varphi ^{\varepsilon }u\right) ,$ then%
\begin{equation*}
\varepsilon ^{2\sigma }\int \left\vert P^{\left( \alpha \right)
}\left( D\right) D^{\beta }\left( \varphi ^{\varepsilon }u\right)
\right\vert ^{2}dx\leq C^{2}\left\vert \left\vert \left\vert
P^{\left( \alpha \right) }\left( D\right) \left( \varphi
^{\varepsilon }u\right) \right\vert \right\vert \right\vert
_{\sigma ,\varepsilon }^{2}
\end{equation*}%
\begin{equation*}
\varepsilon ^{2\sigma }\underset{\alpha \neq 0}{\sum }\varepsilon
^{-2\left\vert \alpha \right\vert }\int \left\vert P^{\left(
\alpha \right) }\left( D\right) D^{\beta }\left( \varphi
^{\varepsilon }u\right) \right\vert ^{2}dx\leq
C^{2}\underset{\alpha \neq 0}{\sum }\varepsilon ^{-2\left\vert
\alpha \right\vert }\left\vert \left\vert \left\vert P^{\left(
\alpha \right) }\left( D\right) \left( \varphi ^{\varepsilon
}u\right) \right\vert \right\vert \right\vert _{\sigma
,\varepsilon }^{2},
\end{equation*}%
consequently lemma 4.6 gives
\begin{equation*}
\varepsilon ^{2\sigma }\underset{\alpha \neq 0}{\sum }\varepsilon
^{-2\left\vert \alpha \right\vert }\int \left\vert P^{\left(
\alpha \right) }\left( D\right) D^{\beta }\left( \varphi
^{\varepsilon }u\right) \right\vert ^{2}dx\leq C\underset{\alpha
\neq 0}{\sum }\varepsilon
^{-2\left\vert \alpha \right\vert }\underset{B_{\varepsilon }}{\int }%
\left\vert P^{\left( \alpha \right) }\left( D\right) \left(
u\right) \right\vert ^{2}dx
\end{equation*}%
As $\varphi ^{\varepsilon }\left( x\right) =\varphi \left( \frac{x}{%
\varepsilon }\right) =1$ in $B_{\frac{\varepsilon }{2}},$ then%
\begin{equation*}
\varepsilon ^{2\sigma }\underset{\alpha \neq 0}{\sum }\varepsilon
^{-2\left\vert \alpha \right\vert }\underset{B_{\frac{\varepsilon }{2}}}{%
\int }\left\vert P^{\left( \alpha \right) }\left( D\right)
D^{\beta }u\right\vert ^{2}dx\leq C\underset{\alpha \neq 0}{\sum
}\varepsilon
^{-2\left\vert \alpha \right\vert }\underset{B_{\varepsilon }}{\int }%
\left\vert P^{\left( \alpha \right) }\left( D\right) \left(
u\right) \right\vert ^{2}dx
\end{equation*}
\end{proof}

\begin{proposition}
Let $\Omega $ be a bounded open set in $\mathbb{R}^{n}$ and $\beta
$ $\in
\mathbb{Z}_{+}^{n}\cap \sigma \mathcal{H}$, then there exists a constant $%
C>0,$ such that for every $u$ solution of $Pu=0$ in $\Omega $ and
$\delta \in \left] 0,1\right[ ,$ we have
\begin{equation*}
\underset{\alpha \neq 0}{\sum }\delta ^{-2\left\vert \alpha
\right\vert }\int_{\Omega _{\delta }}\left\vert P^{\left( \alpha
\right) }\left(
D\right) D^{\beta }u\right\vert ^{2}dx\leq C\delta ^{-2\sigma }\underset{%
\alpha \neq 0}{\sum }\delta ^{-2\left\vert \alpha \right\vert
}\int_{\Omega }\left\vert P^{\left( \alpha \right) }\left(
D\right) u\right\vert ^{2}dx,
\end{equation*}%
where
\begin{equation*}
\Omega _{\delta }=\left\{ x\in \Omega :dist\left( x,\partial
\Omega \right)
>\delta \right\}
\end{equation*}
\end{proposition}

\begin{proof}
The proof is obtained from the precedent lemma and follows the
same reasoning as the proof of theorem 4.4.2 of \cite{HOR2}.
\end{proof}

\begin{corollary}
Let $P\left( D\right) $ an hypoelliptic operator, then $\exists
C>0$ such that for every solution of $P\left( D\right) u=0$ in
$\Omega $,$\forall \varepsilon \in \left] 0,1\right[ ,$\ $\forall
j=1,2,...,$ we have
\begin{equation}
\varepsilon ^{2\sigma }\underset{0\neq \alpha \in \mathbb{N}_{0}^{n}}{\sum }%
\varepsilon ^{-2\left\vert \alpha \right\vert }\left\Vert Q_{\mathcal{H}%
}\left( D\right) P^{\left( \alpha \right) }\left( D\right)
u\right\Vert _{L^{2}\left( \Omega _{\varepsilon j}\right)
}^{2}\leq C\underset{0\neq \alpha \in \mathbb{N}_{0}^{n}}{\sum
}\varepsilon ^{-2\left\vert \alpha \right\vert }\left\Vert
P^{\left( \alpha \right) }\left( D\right) u\right\Vert
_{L^{2}\left( \Omega _{\mathbb{\varepsilon }\left( j-1\right)
}\right) }^{2}
\end{equation}
\end{corollary}

The principal result of this section is the following theorem.

\begin{theorem}
Let $u$ be a solution of the hypoelliptic equation $P\left(
D\right) u=0$ in $\Omega $, then for every $\omega \subset \subset
\Omega ,$ there is a constant $C>0$, such that $\forall j\in
\mathbb{N}$, we have
\begin{equation}
\left\Vert Q_{\mathcal{H}}^{j}\left( D\right) u\right\Vert
_{L^{2}\left( \omega \right) }\leq C^{\left( j+1\right) }j^{\sigma
j}
\end{equation}
\end{theorem}

\begin{proof}
Since $\rho =\rho \left( \omega ,\partial \Omega \right) >0,$ then
there exists $\delta \in \left] 0,\rho \right[ $ such that $\omega
\subset \Omega _{\delta }\subset \Omega .$ Take $\varepsilon
=\dfrac{\delta }{j},$ $j\in \mathbb{N},$ and let us show by
induction on $j$ the following estimate
\begin{equation}
\varepsilon ^{2j\sigma +2m}\underset{0\neq \alpha \in \mathbb{N}_{0}^{n}}{%
\sum }\varepsilon ^{-2\left\vert \alpha \right\vert }\left\Vert \left( Q_{%
\mathcal{H}}^{j}\left( D\right) P^{\left( \alpha \right) }\left(
D\right) u\right) \right\Vert _{L^{2}\left( \Omega _{j\varepsilon
}\right) }^{2}<C^{2\left( j+1\right) },
\end{equation}%
where $m$ is the order of $P$.

As every solution $u$ of an hypoelliptic equation is $\mathcal{C}^{\infty }$%
, then there exists $C>0$ such that $(4.6)$ is satisfied for
$j=0.$ Suppose that $(4.6)$ is true for $j\leq l$ $(l\geq 0),$ we
have to prove that it remains true for $j=l+1.$ Since
$v=Q_{\mathcal{H}}^{l}\left( D\right) u$ is also a solution of
equation $P\left( D\right) u=0$, then from corollary 4.10, we
obtain
\begin{eqnarray}
&&\varepsilon ^{2\sigma \left( l+1\right) +2m}\underset{0\neq \alpha }{\sum }%
\varepsilon ^{-2\left\vert \alpha \right\vert }\left\Vert \left( Q_{\mathcal{%
H}}^{l+1}\left( D\right) P^{\left( \alpha \right) }\left( D\right)
u\right) \right\Vert _{L^{2}\left( \Omega _{\varepsilon \left(
l+1\right) }\right)
}^{2}  \notag \\
&\leq &C\varepsilon ^{2\sigma l+2m}\underset{0\neq \alpha }{\sum }%
\varepsilon ^{-2\left\vert \alpha \right\vert }\left\Vert
P^{\left( \alpha \right) }\left( D\right)
Q_{\mathcal{H}}^{l}\left( D\right) u\right\Vert _{L^{2}\left(
\Omega _{\varepsilon l}\right) }^{2}
\end{eqnarray}%
By the induction hypothesis, we have
\begin{equation*}
\varepsilon ^{2\sigma l+2m}\underset{0\neq \alpha }{\sum
}\varepsilon ^{-2\left\vert \alpha \right\vert }\left\Vert
P^{\left( \alpha \right) }\left( D\right)
Q_{\mathcal{H}}^{l}\left( D\right) u\right\Vert _{L^{2}\left(
\Omega _{\varepsilon l}\right) }^{2}\leq C_{1}^{2\left( l+1\right)
}\;,
\end{equation*}%
consequently, we obtain
\begin{equation*}
\varepsilon ^{2\sigma \left( l+1\right) +2m}\underset{0\neq \alpha }{\sum }%
\varepsilon ^{-2\left\vert \alpha \right\vert }\left\Vert \left( Q_{\mathcal{%
H}}^{l+1}\left( D\right) P^{\left( \alpha \right) }\left( D\right)
u\right) \right\Vert _{L^{2}\left( \Omega _{\varepsilon \left(
l+1\right) }\right) }^{2}\leq C_{2}^{2\left( l+2\right) },
\end{equation*}%
hence $\forall j\in \mathbb{N}$, we have
\begin{equation}
\varepsilon ^{2\sigma j+2m}\underset{0\neq \alpha }{\sum
}\varepsilon ^{-2\left\vert \alpha \right\vert }\left\Vert
P^{\left( \alpha \right) }\left( D\right)
Q_{\mathcal{H}}^{j}\left( D\right) u\right\Vert _{L^{2}\left(
\Omega _{\varepsilon j}\right) }^{2}\leq C_{2}^{2\left( j+1\right)
}
\end{equation}%
The estimate $\left( 4.8\right) $ with $\left\vert \alpha
\right\vert =m$
gives $\forall j\in \mathbb{N},$%
\begin{equation*}
\left\Vert Q_{\mathcal{H}}^{j}\left( D\right) u\right\Vert
_{L^{2}\left( \Omega _{\varepsilon j}\right) }^{2}\leq \varepsilon
^{-2\sigma j}C_{2}^{2\left( j+1\right) },
\end{equation*}%
as $\varepsilon =\dfrac{\delta }{j}$, then
\begin{equation*}
\left\Vert Q_{\mathcal{H}}^{j}\left( D\right) u\right\Vert
_{L_{{}}^{2}\left( \Omega _{\varepsilon j}\right) }^{2}\leq \left( \frac{j}{%
\delta }\right) ^{2\sigma j}C_{2}^{2\left( j+1\right) }\leq
C^{2\left( j+1\right) }j^{2\sigma j},
\end{equation*}%
hence
\begin{equation*}
\left\Vert Q_{\mathcal{H}}^{j}\left( D\right) u\right\Vert
_{L^{2}\left( \Omega _{j\varepsilon }\right) }\leq C^{\left(
j+1\right) }j^{\sigma j}
\end{equation*}
\end{proof}

We denote $G^{s,\text{ }\mathcal{H}}\left( \Omega \right) $ the
multi-anisotropic Gevrey space associated with $\mathcal{H}$ and by $\mu _{%
\mathcal{H}}$ and $\mu _{Q}$ the respective formal orders of the
Newton's polyhedrons $\mathcal{H}$ and $\Gamma \left(
Q_{\mathcal{H}}\right) $, then we have the following relations
\begin{equation*}
\Gamma \left( Q_{\mathcal{H}}\right) =\sigma \mathcal{H}\text{ \
and \ }\mu _{Q}=\sigma \mu _{\mathcal{H}}
\end{equation*}

The principal result of this paper is the following theorem.

\begin{theorem}
Every solution $u\in \mathcal{D}^{\prime }\left( \Omega \right) $
of the
hypoelliptic equation $P\left( D\right) u=0$ is a function of $G^{\frac{%
\sigma }{\mu _{\mathcal{H}}},\;\mathcal{H}}\left( \Omega \right)
.$
\end{theorem}

\begin{proof}
The theorem 4.11 says that every $u$ solution of the hypoelliptic equation $%
Pu=0$ is a Gevrey vector of the operator $Q_{\mathcal{H}}$, i.e. we have $%
u\in G^{\frac{\sigma }{\mu _{Q}}}\left( \Omega
,Q_{\mathcal{H}}\right) .$ From theorem 3.4 and as the operator
$Q_{\mathcal{H}}$ is multi-quasielliptic, then we have $u\in
G^{\frac{\sigma }{\mu _{Q}},\;\Gamma
\left( Q_{\mathcal{H}}\right) }\left( \Omega \right) ,$ and consequently $%
u\in G^{\frac{1}{\mu _{\mathcal{H}}},\;\sigma \mathcal{H}}\left(
\Omega \right) .$ A simple computation shows that in general
$G^{s,\;\sigma \mathcal{H}}\left( \Omega \right) =G^{s\sigma
,\;\mathcal{H}}\Omega ,$ hence $u\in G^{\frac{\sigma }{\mu
_{\mathcal{H}}},\;\mathcal{H}}\left( \Omega \right) .$
\end{proof}

\begin{remark}
It is interesting to compare the result of the theorem with the
microlocal Gevrey regularity result obtained in \cite{BC3}.
\end{remark}


\begin{thebibliography}{99}
\bibitem{BC2} C. Bouzar, R. Cha\"{\i}li, Gevrey vectors of
multi-quasielliptic systems, Proc. Amer. Math. Soc. 131 (5)
(2003), 1565-1572

\bibitem{BC3} C. Bouzar, R. Cha\"{\i}li, A Gevrey microlocal analysis of
multi-anisotropic differential operators. Rend. Sem. Mat. Univ.
Pol. Torino, Vol. 64, 3, (2006), 305-318

\bibitem{FRIB} J. Friberg, Multi-quasielliptic polynomials, Ann. Sc. Norm.
Sup. Pisa. Cl. di. Sc. 21 (1967), 239-260.

\bibitem{GV} S. G. Gindikin, L. R. Volevich, The method of Newton polyhedron
in the theory of partial differential equations, Kluwer, (1992)

\bibitem{HAK} G. H. Hakobyan, Estimates of the higher order derivatives of
the solution of hypoelliptic equations. Rend. Sem. Mat. Univ. Pol.
Torino, Vol. 61:4, (2003), 443--459.

\bibitem{HOR} L. H\"{o}rmander, Distributions theory and Fourier analysis,
Springer, (1990).

\bibitem{HOR2} L. H\"{o}rmander, Linear partial differential operators,
Springer (1969).

\bibitem{MIKH} V. P. Mikha\"{\i}lov, The behavoir at infinity of a class of
polynomials, Proc. Steklov. Inst. Mat. 91 (1967), 59-80.

\bibitem{ROD} L. Rodino, Linear partial differential operators in Gevrey
spaces, World Scientific, (1993).

\bibitem{ZANG} L. Zanghirati, Iterati di una classe di operatori
ipoelliptici e classi generalizzate di Gevrey, Suppl. Boll. U.M.I.
(1980), 177-195.
\end{thebibliography}
\end{document}